\documentclass[a4paper,12pt]{amsart}
\usepackage{amssymb,amsfonts,amsmath}
\usepackage{ifthen}
\usepackage{graphicx}
\usepackage{float}
\newtheorem{theorem}{Theorem}[section]
\newtheorem{corollary}[theorem]{Corollary}
\newtheorem{lemma}[theorem]{Lemma}

\newtheorem{remark}[theorem]{Remark}

\theoremstyle{definition}
\newtheorem{definition}{Definition}[section]

\newcounter{minutes}
\divide\time by 60
\newcounter{hours}
\multiply\time by 60
\addtocounter{minutes}{-\time}

\newcommand{\D}{{\mathbb D}}
\newcommand{\IC}{{\mathbb C}}
\newcommand{\real}{{\operatorname{Re}\,}}

\newcommand{\ds}{\displaystyle}

\begin{document}

\bibliographystyle{amsplain}

\title[Nehari's univalence criteria and pre-Schwarzian derivative]%
{Nehari's univalence criteria, pre-Schwarzian derivative and applications}

\def\thefootnote{}
\footnotetext{ \texttt{\tiny File:~\jobname .tex,
          printed: \number\day-\number\month-\number\year,
          \thehours.\ifnum\theminutes<10{0}\fi\theminutes}
} \makeatletter\def\thefootnote{\@arabic\c@footnote}\makeatother

\author{Sarita Agrawal${}^*$}
\address{Sarita Agrawal, Institute of Mathematics and Applications,
Andharua, Bhubaneswar 751029, Odisha, India}
\email{saritamath44@gmail.com}

\author{Swadesh Kumar Sahoo}
\address{Swadesh Kumar Sahoo, Discipline of Mathematics,
Indian Institute of Technology Indore,
Simrol, Khandwa Road, Indore 453 552, India}
\email{swadesh.sahoo@iiti.ac.in}

\thanks{${}^*$ The corresponding author}

\begin{abstract}
In this paper we study sharp estimates of pre-Schwarzian derivatives of functions belonging to the Nehari-type classes by using techniques from
differential equations. 
In the sequel, we also see that a solution of a complex 
differential equation has a special form in terms of
ratio of hypergeometric functions resulting to an integral representation.
Finally, we attempt to study those univalent functions in the unit disk for which the image domain is an unbounded John domain.\\


\end{abstract}

\maketitle
\pagestyle{myheadings}
\markboth{S. Agrawal and S. K. Sahoo}{Nehari's univalence criteria and pre-Schwarzian derivative}

\section{Introduction and preliminaries}\label{Sec-Intro}
Let $\D:=\{z\in \mathbb{C}:\,|z|<1\}$ be the open unit disk in the complex plane $\mathbb{C}$.
We denote by $\overline{\IC}:=\mathbb{C}\cup\{\infty\}$ the extended complex plane or equivalently the {\em Riemann sphere}.
The {\em Schwarzian derivative} of a locally injective meromorphic 
function $f:\D\to \overline{\IC}$ is defined by
$$S_f(z)=T'_f(z)-\frac{1}{2}T^2_f(z)
$$
at each point $z$ where $f$ is analytic, and $S_f(z)=S_{1/f}(z)$ at the poles of $f$.
Here, the quantity $T_f(z)={f''(z)}/{f'(z)}$ is known as the {\em pre-Schwarzian derivative} of $f$ 
or the logarithmic derivative of $f'$.  
We denote by $\mathcal{A}$, the class of all analytic  
functions in $\D$ normalized so that $f(0)=0=f^\prime(0)-1$.
The set $\mathcal{S}$ denotes the class of univalent functions in $\mathcal{A}$. 
There are many sufficient conditions available in the literature for a function 
to be univalent in $\D$ and most of them are 
very far from necessary conditions. 
However, there are a few of them which are also close to necessity. 
One such example is about the well-known Nehari criterion. 
From this fact, the Nehari class is generated and it is associated with the 
Schwarzian derivative of functions (see \cite{Dur83,Neh49,Neh54}).
Moreover, sufficient conditions for starlikeness and convexity in terms of Schwarzian 
derivatives are studied in \cite{KS}.
For $\alpha\ge0$ and $k\ge 0$, we set
\begin{equation}\label{Nehari-eqn}
\mathcal{N}_{\alpha}(k)=\left\{f\in\mathcal{A}:(1-|z|^2)^\alpha|S_f(z)|\le k,
~f''(0)=0 \right\}.
\end{equation}
The set $\mathcal{N}_2(2)$, called the Nehari class, is intensively studied by Chuaqui, Osgood and Pommerenke in \cite{COP96}. 
Due to \cite[Lemma~1]{CO94},
if $f\in \mathcal{N}_2(k)$ then $(1-|z|^2)|T_f(z)|\le k|z|$, for $0\le k\le 2$.
However, the constant $k$ in this case is not best possible. This result is indeed improved and
discussed in Section~\ref{sec2} of this paper.
Except for the cases $\alpha=2, k=2; \alpha=1, k=4 \mbox{ and } \alpha=0, k=\pi^2/2$, all mapping considered in the Schwarzian classes $\mathcal{N}_{\alpha}(k)$ have images that are quasidisks, that is, John disks whose complements are also John disks \cite{NV91} (See Section~\ref{sec4} for the definition of John disks). It follows from \cite[Theorem~6]{GP84}, that if $|S_f(z)|\le \rho(z)$ is a sufficient condition for univalence in the disk, then $|S_f(z)|\le t\rho(z)$ for some $0\le t<1$ which guarantees that the images are quasidisks. Furthermore, in the cases $\alpha=1, k=4 \mbox{ and } \alpha=0, k=\pi^2/2$, the images will also be quasidisks as soon as they are Jordan domains.  

We conclude this section by providing short introduction about upcoming sections.
In Section~\ref{sec2}, we prove sharp estimates for the pre-Schwarzian 
derivatives for functions in Nehari-type classes. An estimate related to the Schwarzian derivative is established in Section~\ref{sec3}.
Concluding remarks and future directions are discussed in Section~\ref{sec4}.

\section{Estimates of the pre-Schwarzian derivative}\label{sec2}
This section deals with sharp estimates of pre-Schwarzian derivative
of functions $f$ belonging to the families $\mathcal{N}_0(k)$, $\mathcal{N}_1(k)$, and 
$\mathcal{N}_2(k)$.
A technique from differential equations is used in estimating the pre-Schwarzian derivatives. 
First, note that the family $\mathcal{N}_2(k)$ has got special attractions in geometric
function theory in compare to the other two families (see \cite[pp.~261--264]{Dur83}). 
It was initially pointed out by Kraus \cite{Kra32}
in 1932 and was rediscovered by Nehari \cite{Neh49} in 1949 that 
$|S_f(z)|\le 6(1-|z|^2)^{-2}$, $|z|<1$, for each $f\in\mathcal{S}$.
In the same paper Nehari found its converse counterpart which says that
for an $f\in\mathcal{A}$ if $|S_f(z)|\le 2(1-|z|^2)^{-2}$, 
then $f\in \mathcal{S}$. These results are concerning 
the functions related to the family $\mathcal{N}_2(2)$.
Of course, definition of $\mathcal{N}_\alpha(k)$ also involves a normalization
which is essential to prove our main results. 
In addition to this, the same paper of Nehari
also deals with the inequality $|S_f(z)|\le \pi^2/2$ which 
is sufficient for univalence. So, this is related to the family 
$\mathcal{N}_0(\pi^2/2)$. In a similar vein, Pokornyi \cite{Pok51}
stated that the condition $|S_f(z)|\le 4(1-|z|^2)^{-1}$ is sufficient
for univalence and Nehari \cite{Neh54} supplied its proof. Certainly,
functions from the family $\mathcal{N}_1(k)$ satisfy the last inequality
with the constant $k=4$. These are some of the reasons for which we present 
our results starting with the family $\mathcal{N}_2(k)$ followed by 
$\mathcal{N}_0(k)$ and $\mathcal{N}_1(k)$, respectively.

\medskip
It is now appropriate to recall

\subsection*{Gr\"onwall's Inequality}\cite[p.~241]{CC97}
Let $I$ denote an interval of the real line of the form $[a,\infty)$ or 
$[a,b]$ or $[a,b)$ with $a<b$. Let $\beta$ and $u$ be real valued
continuous functions defined on $I$. If $u$ is differentiable in the 
interior $I^0$ of $I$ (the interval without the end points $a$ and
possibly $b$) and satisfies the differential inequality
$$u'(t)\le u(t)\beta(t), \quad t\in I^0,
$$  
then $u$ is bounded by the solution of the differential equation
$u'(t)=u(t)\beta(t)$: 
$$u(t)\le u(a)\exp\left(\int_a^t\beta(s)\,ds\right)
$$
for all $t\in I$.
The following result is a generalization of \cite[Lemma~1]{CO94}. Note that this idea was originally proposed by Chuaqui and Osgood (see \cite[pp.~660-662]{CO94}), 
but it was not precisely estimated whereas an optimal bound for $|T_f|$, $f\in \mathcal{N}_2(k)$, 
was proved. Here we provide the sharp estimation of $|T_f|$, $f\in \mathcal{N}_2(k)$, precisely.
\begin{theorem}\label{newlemma}
If $f\in \mathcal{N}_2(k)$, $0\le k\le 2$, then 
\begin{equation}\label{newlemma-e1}
\left|T_f(z)\right|\le \frac{2|z|-2\beta^2 A_k(|z|)}{1-|z|^2},
\end{equation}
where $\ds A_k(z)=\frac{1}{\beta}\frac{(1+z)^\beta-(1-z)^\beta}{(1+z)^\beta+(1-z)^\beta}$
with $\beta=\sqrt{1-(k/2)}$. 
Equality holds at a single $z\neq 0$ if and only if $f$ 
is a suitable rotation of $A_k(z)$. Here, $(1\pm z)^\beta$ represents the principal powers
so that $(1\pm z)^\beta=\exp(\beta\log(1\pm z))$ are analytic in $\mathbb{D}$.
\end{theorem}
\begin{proof}
A simple computation gives
$$T_f'(z)=\frac{1}{2}T_f^2(z)+S_f(z),\quad T_f(0)=0.
$$
Now, consider the initial value problem
$$w^{\prime}(x)=\frac{1}{2}w^2(x)+\frac{k}{(1-x^2)^2},\quad w(0)=0
$$
on $(-1,1)$. Note that it is satisfied by $\ds w(x)=\frac{2x-2\beta^2 A_k(x)}{1-x^2}$.
We shall show that $|T_f(z)|\le w(|z|)$.

Fix $z_0$ with $|z_0|=1$, and let
$$\psi(\tau)=|T_f(\tau z_0)|,\quad 0\le \tau<1.
$$
It is evident that the zeros of $\psi(\tau)$ are isolated unless $f(z)\equiv z$. Away from these zeros, $\psi(\tau)$ is 
differentiable and satisfies $\psi'(\tau)\le|T_f^{\prime}(\tau z_0)|$. Since $(1-\tau^2)^2|S_f (\tau z_0)|\le k$ 
we obtain
$$\frac{d}{d\tau}(\psi(\tau)-w(\tau))\le|T_f^{\prime}(\tau z_0)|-w^{\prime}(\tau)
\le\frac{1}{2}(|T_f(\tau z_0)|^2-w^2(\tau))=\frac{1}{2}(\psi(\tau)-w(\tau))(\psi(\tau)+w(\tau)).
$$
The initial condition $\psi(0)-w(0)=0$, with the Gr\"onwall inequality, tells us that 
$\psi(\tau)-w(\tau)\le 0$
and hence the required inequality follows.

For the equality part, one can easily see that $A_k(z)\in \mathcal{N}_2(k)$ and the equality
$$\left|T_f(z)\right|=\frac{2|z|-2\beta^2 A_k(|z|)}{1-|z|^2}
$$
holds for $0\neq z\in \D$, with a suitable rotation ($\theta=-\arg(z)$) of $A_k(z)$. 
Conversely, suppose that equality holds in \eqref{newlemma-e1} 
for some $0\neq z_1\in \D$. Fix $z_0=z_1/|z_1|$ and $\psi(\tau)$ as defined above.
Then $\psi(|z_1|)=w(|z_1|)$, which can happen only if $\psi(\tau)=w(\tau)$ for
all $\tau\in [0,1)$. Hence, $T_f(\tau z_0)=e^{i\theta} w(\tau)$. From this, it follows that for all $z\in \D$, $T_f(z)=cw(z\overline{z_0})$ with $|c|=1$. 
Integration of both the sides with a suitable rotation ($c=\overline{z_0}$) and some simplification yields
$$f'(z)=\frac{(1-z^2 \overline{z_0}^2)^{\beta-1} K_1}{((1+z\overline{z_0})^\beta+(1-z\overline{z_0})^\beta)^2},
$$
where $K_1$ is the constant of integration. Appealing to the normalization condition we get $K_1=4$. Again integrating we get
$$f(z)=\overline{c} A_k(cz)
$$
as required.
\end{proof}
We observe that $w(x)=(2x+\sqrt{4-2k})/(1-x^2)$ is a solution of the 
differential equation  
$$w^{\prime}(x)=\frac{1}{2}w^2(x)+\frac{k}{(1-x^2)^2}
$$
on $(-1,1)$. Also, note that $w(0)=\sqrt{4-2k}$.  This motivates us   
to define a class similar to $\mathcal{N}_2(k)$ with
a normalization in the following way:
$$
\mathcal{M}_2(k)=\left\{f\in\mathcal{A}:(1-|z|^2)^2|S_f(z)|\le k,
~f''(0)=\sqrt{4-2k}\right\}, \quad 0\le k\le 2 .
$$
Note that if $k=2$, the class $\mathcal{M}_2(k)$ coincides with the Nehari class 
$\mathcal{N}_2(2)$. Now we give a result similar to Theorem~\ref{newlemma} for the class $\mathcal{M}_2(k)$
where the bound obtained is more simpler than Theorem~\ref{newlemma}.
\begin{lemma}\label{l1}
If $f\in \mathcal{M}_2(k)$, $0\le k\le 2$, then 
$$\left|T_f(z)\right|\le \frac{2|z|+\sqrt{4-2k}}{1-|z|^2}.
$$
Equality holds at a single $z\neq 0$ if and only if $f$ is a suitable rotation of $F_0(z)$, where
$$F_0(z)=\frac{e^{\displaystyle\sqrt{4-2k}\tanh^{-1}(z)}-1}{\sqrt{4-2k}}
=\frac{\ds\left(\frac{1+z}{1-z}\right)^b-1}{2b},
$$
where $b=\sqrt{4-2k}/2$.
Here, a suitable branch for $[(1+z)/(1-z)]^b$ is chosen 
so that $[(1+z)/(1-z)]^b=\exp(b\log[(1+z)/(1-z)])$ becomes analytic in $\mathbb{D}$.
\end{lemma}
\begin{figure}[H]
\includegraphics[width=5cm]{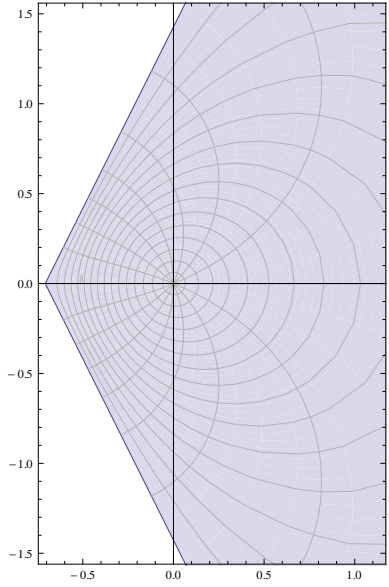}
\caption{Graph of the function $F_0(z)$ with $k=1$.} 
\end{figure}
\begin{proof}
A simple computation gives
$$T_f'(z)=\frac{1}{2}T_f^2(z)+S_f(z),\quad T_f(0)=\sqrt{4-2k}.
$$
Now, consider the initial value problem
$$w^{\prime}(x)=\frac{1}{2}w^2(x)+\frac{k}{(1-x^2)^2},\quad w(0)=\sqrt{4-2k}
$$
on $(-1,1)$. Note that it is satisfied by $w(x)=(2x+\sqrt{4-2k})/(1-x^2)$.
Now it is enough to prove that $|T_f(z)|\le w(|z|)$. This can be proved similar to the proof given in Theorem~\ref{newlemma}.
Finally, one can easily see that $F_0(z)\in \mathcal{M}_2(k)$ and the equality
$$\left|T_{F_0}(z)\right|=\frac{2|z|+\sqrt{4-2k}}{1-|z|^2}
$$
holds for $0\neq z\in \D$, with a suitable rotation of $F_0(z)$. 
The proof of only if part follows in the similar manner as in the proof of that of Theorem~\ref{newlemma}. 
\end{proof}
If we choose $k=2$ in Theorem~\ref{newlemma} and Lemma~\ref{l1}, we obtain the following well-known result.
\begin{corollary}\cite[Lemma~1]{CO94}\label{c1}
If $f\in \mathcal{N}_2(2)$ then 
$$\left|T_f(z)\right|\le \frac{2|z|}{1-|z|^2}.
$$
Equality holds at a single $z\neq 0$ if and only if $f$ is a rotation of
$$\frac{1}{2}\ln\frac{1+z}{1-z}.
$$
\end{corollary}
Similarly, the next result is stated as follows:
\begin{theorem}\label{sec2-lem4}
If $f\in \mathcal{N}_0(k)$, $0\le k\le \pi^2/2$, then 
$$\left|T_f(z)\right|\le \sqrt{2k}\tan\left(\sqrt{\frac{k}{2}}|z|\right).
$$
Equality holds at a single $z\neq 0$ if and only if $f$ is a rotation of $F_1(z)$, where
$$F_1(z)=\sqrt{\frac{2}{k}}\tan\left(\sqrt{\frac{k}{2}}z\right).$$
\end{theorem}
\begin{proof}
A simple computation gives
$$T_f^{\prime}(z)=\frac{1}{2}T_f^2(z)+S_f(z),\quad T_f(0)=0.
$$
Now, consider the initial value problem 
$$w^{\prime}(x)=\frac{1}{2}w^2(x)+k,\quad w(0)=0
$$
on $(-1,1)$. Clearly $w(x)=\sqrt{{2}/{k}}\tan\Big(\sqrt{{k}/{2}}\,x\Big)$ 
is a solution of the initial value problem.
Now it remains to show that $|T_f(z)|\le w(|z|)$ which follows in the similar way given in the proof of Theorem~\ref{newlemma}. 
One can easily see that the equality holds for $F_1(z)$ defined in the statement of the lemma and can be proved in the same technique used in Theorem~\ref{newlemma}.
\end{proof}
\begin{figure}[H]
\includegraphics[width=5cm]{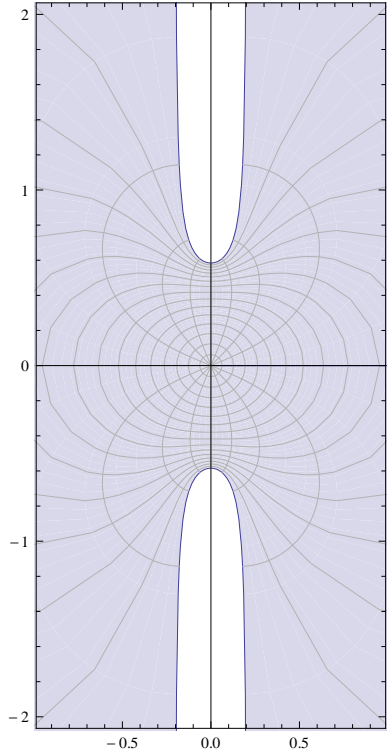}
\caption{Graph of the function $F_1(z)$ with $k=\pi^2/2$.} 
\end{figure}

\begin{corollary}\label{cor2.4}
If $f\in \mathcal{N}_0(\pi^2/2)$ then 
$$
\left|T_f(z)\right|\le\pi\tan\left({\frac{\pi}{2}|z|}\right).
$$
The equality holds at a single $z\neq 0$ if and only if $f$ is a rotation of
$$\frac{2}{\pi}\tan\left(\frac{\pi}{2}z\right).
$$
\end{corollary}

Next we present a similar result for functions in the class $\mathcal{N}_1(k)$. Since we use the same technique
and it involves solution of a differential equation, as a supplementary result we see that the solution of the
differential equation is a Gaussian hypergeometric function. 
We denote by $F(a,b;c;z)$ the Gaussian hypergeometric function
defined by
$$F(a,b;c;z)=\sum_{n=0}^\infty\frac{(a)_n(b)_n}{(c)_n(1)_n}z^n,\quad
z\in\D,
$$
where $(a)_0=1$, $(a)_n=a(a+1)\cdots(a+n-1)$ is the Pochhammer symbol
and $c\neq 0,-1,-2,\ldots$. We have the well-known derivative formula
$$F'(a,b;c;z)=\frac{d}{dz}F(a,b;c;z)= \frac{ab}{c}F(a+1,b+1;c+1;z).
$$
\begin{theorem}\label{sec2-thm6}
A solution of the differential equation 
$$w^{\prime}(z)=\frac{1}{2}w^2(z)+\frac{k}{1-z^2},\quad 0\le k\le 4,
$$
can be represented by $\ds w(z)=\int_0^1 \frac{kz}{1-tz^2} d\mu(t)$, 
where $\mu(t):[0,1]\to[0,1]$ is a non- decreasing function with $\mu(1)-\mu(0)=1$. 
In particular, $|w(z)|\le w(|z|)$.
\end{theorem}
\begin{proof}
The result is trivial if $k=0$. Now assume that $0<k\le 4$.
Let a solution of the differential equation 
$$w^{\prime}(z)=\frac{1}{2}w^2(z)+\frac{k}{1-z^2}
$$
be of the form $w(z)={-2u^{\prime}(z)}/{u(z)}$. Then $u(z)$ is a solution of the 
second order linear differential equation
$$u^{\prime\prime}+\frac{k}{2(1-z^2)}u=0.
$$
It can easily be verified that this differential equation is satisfied by
$$u(z)=F[(-1/4) (1+\sqrt{1+2k}), (1/4) (-1+\sqrt{1+2k}); 1/2; z^2], \quad |z|<1.
$$
Note that the series solution method can also produce two linearly independent 
solutions where the above 
hypergeometric representation of $u(z)$ is one of them. 
Hence, the required solution is
\begin{align*}
w(z) & =kz\left(\frac{F[(-1/4) (-3+\sqrt{1+2k}), (1/4) (3+\sqrt{1+2k}); 3/2; z^2]}{F[(-1/4) 
(1+\sqrt{1+2k}), (1/4) (-1+\sqrt{1+2k}); 1/2; z^2]}\right)\\
&=\int_0^1 \frac{kz}{1-tz^2} d\mu(t)
\end{align*}
for a non-decreasing function $\mu:[0,1]\to [0,1]$ with $\mu(1)-\mu(0)=1$,
where the second equality follows from \cite[Theorem~1.5]{Kus02} (see also 
\cite[Lemma~7]{CKPS05}) since $0<k\le 4$.

Finally, it follows that 
$$
|w(z)|\le \int_0^1 \frac{k|z|}{|1-tz^2|} d\mu(t)\le \int_0^1 \frac{k|z|}{1-t|z|^2} d\mu(t)=w(|z|),
$$
completing the proof.
\end{proof}

Now we can estimate the pre-Schwarzian derivative of a function $f$ in $\mathcal{N}_1(k)$.

\begin{theorem}\label{sec2-lem7}
If $f\in \mathcal{N}_1(k)$, $0\le k\le 4$, then 
$$\left|T_f(z)\right|\le k|z|\left(\frac{F[(-1/4) (-3+\sqrt{1+2k}), 
(1/4) (3+\sqrt{1+2k}); 3/2; |z|^2]}{F[(-1/4) (1+\sqrt{1+2k}), (1/4) (-1+\sqrt{1+2k}); 1/2; |z|^2]}\right).
$$
Equality holds at a single $z\neq 0$ if and only if $f$ is a rotation of $F_2(z)$, where
$$F_2(z)=\int_0^z\frac{1}{(F[(-1/4) (1+\sqrt{1+2k}), (1/4) (-1+\sqrt{1+2k}); 1/2; t^2])^2}\, dt.
$$
\end{theorem}
\begin{proof}
An easy computation gives that
$$T_f^{\prime}(z)=\frac{1}{2}T_f^2(z)+S_f(z),\quad T_f(0)=0.
$$
Consider the initial value problem
$$w^{\prime}(x)=\frac{1}{2}w^2(x)+\frac{k}{1-x^2},\quad w(0)=0
$$
on $(-1,1)$. Use Theorem~\ref{sec2-thm6} and proceed 
in the same manner as in the proof of Theorem~\ref{sec2-lem4}. We can easily show that $|T_f(z)|\le w(|z|)$. 
Equality can also be verified easily by considering the function $F_2(z)$ defined in the statement.
\end{proof}

\begin{figure}[H]
\includegraphics[width=5cm]{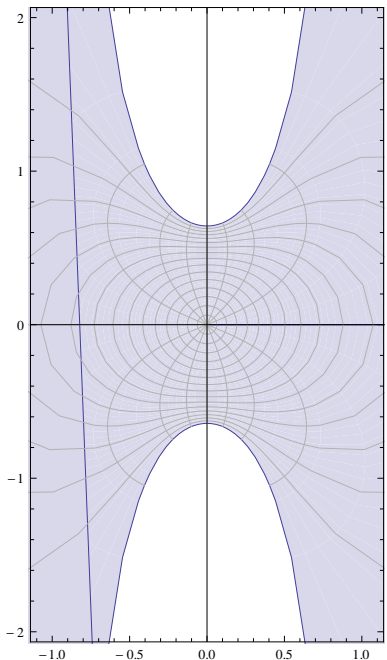}
\caption{Graph of the function $F_2(z)$ with $k=4$.} 
\end{figure}

As a consequence of Theorem~\ref{sec2-lem7}, we obtain
\begin{corollary}\label{sec2-cor8}
If $f\in \mathcal{N}_1(4)$ then 
$$\left|T_f(z)\right|\le \frac{4|z|}{1-|z|^2}.
$$
Equality holds at a single $z\neq 0$ if and only if $f$ is a rotation of
$$\frac{1}{4}\left(\frac{2z}{1-z^2}+\ln\frac{1+z}{1-z}\right).
$$
\end{corollary}

\section{Schwarzian derivative and John domains}\label{sec3}
This section is devoted to the study of functions in Nehari-type classes.
We begin with the definition of John domain. {\em John domains} in the Euclidean 
$n$-space $\mathbb{R}^n$ which were introduced by John \cite{Jo61} 
in connection with his work on elasticity. 
The term ``John domain" is due to Martio and Sarvas \cite{MS78}
while studying certain injectivity theorems for functions defined on some special plane domains
in terms of the Schwarzian 
derivative (see for instance \cite[Theorem~4.14 and Theorem~4.24]{MS78}).
Bounded John domains are characterized by the following geometric fact: a bounded domain 
$D\subset\mathbb{C}$ is a John domain if and only if there is a 
constant $a>0$ such that for every (straight) {\em crosscut} $C$ of $D$ the inequality
$${\rm diam}\, H\le a\, {\rm diam}\, C
$$
holds for one of the components $H$ of $D\setminus C$. 
Here ``{\rm diam}'' denotes the Euclidean diameter. 

We shall take help of the following well-known characterization to find necessary conditions 
for $f(\D)$ to be John domains while the function $f$ belongs to the Nehari-type families
$\mathcal{N}_\alpha(k)$ defined in the first section.

\begin{lemma}\cite[Lemma~2]{COP96}\label{cop-l2}
Let $f$ be analytic and univalent in $\D$. Then $f(\D)$ is a John domain 
if and only if there exists $0<x<1$ such that
$$ \sup_{|\zeta|=1} \sup_{r<1} \frac{(1-\rho^2)|f^\prime(\rho\zeta)|}{(1-r^2)|f^\prime(r\zeta)|}<1, 
\quad \rho=\frac{x+r}{1+xr}.
$$
\end{lemma}

\medskip
This characterization plays an important role to prove the following necessary condition
for bounded John domains $f(\mathbb{D})$, when $f\in \mathcal{N}_2(2)$ (see \cite[Theorem~4]{COP96}).
In fact, in its proof,
a relationship between the derivatives $S_f$ and $T_f$ as well as an upper bound for 
$|T_f|$ are used.

\begin{lemma}\cite[Theorem~4]{COP96}\label{l0}
Let $f\in \mathcal{N}_2(2)$ and $f(\D)$ be a bounded John domain. Then
$$\limsup_{|z|\to 1}(1-|z|^2)\real(zT_f(z))<2.
$$
\end{lemma}

\medskip
Naturally, one can ask the analog of Lemma~\ref{l0} for the family 
$\mathcal{N}_2(k)$, $0\le k< 2$.
From \cite[Lemma~1]{CO94} it is clear that for all bounded mappings,
$$\limsup_{|z|\to 1}(1-|z|^2)^2|S_f(z)|\le k \implies \limsup_{|z|\to 1}(1-|z|^2)|T_f(z)|\le k.
$$
From this, we conclude that, for $f\in \mathcal{N}_2(k)$, $0\le k\le 2$,
$$\limsup_{|z|\to 1}(1-|z|^2)\real(zT_f(z))\le k.
$$ 
By the similar argument, from Theorem~$\ref{newlemma}$, we can say that
$$\limsup_{|z|\to 1}(1-|z|^2)\real(zT_f(z))\le 2-\sqrt{4-2k},
$$ 
for all bounded mappings in $\mathcal{N}_2(k)$, $0\le k\le 2$. 
Here the bound $2-\sqrt{4-2k}$ improves the bound $k$ but the same proof method given in \cite[Theorem~4]{COP96} is not working to get the exact analog of Lemma~\ref{l0} 
for the class $\mathcal{N}_2(k)$, $0\le k<2$.

\medskip
An analog of Lemma~$\ref{l0}$ for the class $\mathcal{M}_2(k)$ is now presented here. We use the similar technique as used to prove Lemma~$\ref{l0}$.  
\begin{theorem}\label{thm2.1}
Let $f\in \mathcal{M}_2(k)$, $0\le k< 2$. Then 
$$\limsup_{|z|\to 1}(1-|z|^2)\real(zT_f(z))<2+\sqrt{4-2k}.
$$
\end{theorem}
\begin{proof}
From \cite[Theorem~6]{GP84}, it is clear that $f(\D)$ is a John domain.
By Lemma \ref{l1} we get
$$|T_f(z)|\le \frac{2|z|+\sqrt{4-2k}}{1-|z|^2},
$$
and, with $|z|=r$,
\begin{eqnarray*}
|T_f^{\prime}(z)|&=&\left|S_f(z)+\frac{1}{2}T_f^2(z)\right|\\
&\le & \frac{k}{(1-r^2)^2}+\frac{1}{2}\left(\frac{2r+\sqrt{4-2k}}{1-r^2}\right)^2\\
&=&\frac{d}{dr}\left(\frac{2r+\sqrt{4-2k}}{1-r^2}\right).
\end{eqnarray*}
We prove the theorem by contradiction method.
Suppose that the required inequality does not hold. That is, $\exists$ a sequence $z_m\in \D$ with $|z_m|\to 1$ such that
$$\limsup_{|z_m|\to 1}(1-|z_m|^2)\real(z_mT_f(z_m))\ge 2+\sqrt{4-2k}.
$$ 
Now choose a subsequence $z_{m_l}(=z_n)$ of $z_m$ with $|z_n|\to 1$ such that
\begin{equation}\label{e1}
(1-|z_n|^2)\real\{z_nT_f(z_n)\}\to 2+\sqrt{4-2k}
\end{equation}
holds.
Let $x\in (0,1)$ be fixed. Set $z_n=\rho_n\zeta_n$, $|\zeta_n|=1$, and $r_n=(\rho_n-x)/(1-x\rho_n)$. 
The above upper bound for $T_f'$ leads to
$$ |\real\{\zeta_nT_f(z_n)\}-\real\{\zeta_nT_f(r\zeta_n)\}|\le \int_r^{\rho_n}|T_f^\prime(t\zeta_n)|dt
\le \frac{2\rho_n+\sqrt{4-2k}}{1-\rho_n^2}-\frac{2r+\sqrt{4-2k}}{1-r^2}
$$
or,
$$-\real\{\zeta_nT_f(r\zeta_n)\}\le \frac{2\rho_n+\sqrt{4-2k}}{1-\rho_n^2}
-\frac{2r+\sqrt{4-2k}}{1-r^2}-\real\{\zeta_nT_f(z_n)\}
$$
or,
$$\hspace*{-7cm}-\frac{1-r^2}{r}\real\{\zeta_nT_f(r\zeta_n)\}\le -\frac{\sqrt{4-2k}+2r}{r}
$$
$$\hspace*{5cm}+\frac{1-r^2}{1-\rho_n^2}\frac{1}{r\rho_n}\left[(\sqrt{4-2k}+2\rho_n)
\rho_n-(1-\rho_n^2)\real \{z_nT_f(z_n)\}\right].
$$
If $r_n\le r \le \rho_n$ then
$$ \frac{1-r^2}{1-\rho_n^2}\le \frac{1+x}{1-x}
$$
and
$$ -\frac{\sqrt{4-2k}+2r}{r} \le -\frac{\sqrt{4-2k}+2\rho_n}{\rho_n}.
$$
Hence,  
\begin{align*}
& (2+\sqrt{4-2k})-\frac{1-r^2}{r}\real\{\zeta_nT_f(r\zeta_n)\}\\
& \hspace*{2cm} \le \left( \frac{1+x}{1-x}\right) 
\left(2+\sqrt{4-2k}-(1-|z_n|^2)\real\{z_nT_f(z_n)\}\right).
\end{align*}
Therefore, by the assumption (\ref{e1}) we get
$$\left|(2+\sqrt{4-2k})-\frac{1-r^2}{r}\real\{\zeta_nT_f(r\zeta_n)\} \right|< \epsilon
$$
for all $n\ge n_0(\epsilon,x)$.

From the above estimations we get
\begin{eqnarray*}
\log \frac{(1-r_n^2)|f^\prime(r_n\zeta_n)|}{(1-\rho_n^2)|f^\prime(\rho_n\zeta_n)|} &=& \int_{r_n}^{\rho_n} 
\left(\frac{2r}{1-r^2}-\real\{\zeta_nT_f(r\zeta_n)\right)dr\\
&<& \int_{r_n}^{\rho_n} \frac{\epsilon}{1-r^2}dr-\sqrt{4-2k} \int_{r_n}^{\rho_n} \frac{r}{1-r^2} dr\\
&<& \int_{r_n}^{\rho_n} \frac{\epsilon}{1-r^2}dr = \epsilon h_{\D}(r_n\zeta_n,\rho_n\zeta_n)=\epsilon h_{\D}(0,x),
\end{eqnarray*}
for $n\ge n_0$. Here, $h_{\D}(\cdot,\cdot)$ denotes the usual hyperbolic distance of the unit disk $\D$. 
Thus,
$$\frac{(1-\rho_n^2)|f^\prime(\rho_n\zeta_n)|}{(1-r_n^2)|f^\prime(r_n\zeta_n)|}> e^{-\epsilon h_{\D}(0,x)}.
$$
But since $\rho_n=(r_n+x)/(1+xr_n)$, the last inequality contradicts to Lemma \ref{cop-l2}.
\end{proof}
\begin{remark}
One can ask similar questions when $f\in \mathcal{N}_0(k)$, $0\le k\le \pi^2/2$ and 
$f\in  \mathcal{N}_1(k)$, $0\le k\le 4$. Indeed, we notice that in these cases the quantity
$$\limsup_{|z|\to 1}(1-|z|^2)\real\left\{zT_f(z)\right\}
$$
vanishes due to \cite[Lemma~1]{CO94}. 
\end{remark}

\section{Concluding Remarks}\label{sec4}
Recall that N\"akki and V\"ais\"al\"a in \cite{NV91} 
introduced the notion of John domains when they are unbounded
and also studied several characterizations of such domains.
According to them, John domains are defined as follows:

\begin{definition}\label{JohnDomain}
A domain $D\subset \mathbb{C}$ is said to be a {\em John domain} if any pair of points
$z_1,z_2\in D$ can be joined by a rectifiable path $\gamma\subset D$ such that
$$\min\{\ell(\gamma[z_1,z]),\ell(\gamma[z,z_2])\}\le c\,{\rm dist}\,(z,\partial D), 
\quad \mbox{ for all $z\in \gamma$},
$$
and for some constant $c>0$, where $\ell(\gamma[z,z_i])$ denote the Euclidean length of 
$\gamma$ joining $z$ to $z_i$, $i=1,2$. 
\end{definition}

Note that a simply connected John domain is called a {\em John disk} and when 
John domains are bounded, then Definition~\ref{JohnDomain} is equivalent to 
the definition of John domains discussed in Section~\ref{sec3} (see \cite{NV91}).
One can check that the parallel strip $D_1:=\{z\in\mathbb{C}:\,|{\rm Im}\,z|<\pi/4\}$ and 
the two-sided slit domain $D_2$, the entire plane minus the two half-lines $-\infty<y\le -1/2$ and 
$1/2\le y<\infty$, $y={\rm Im}\,z$, are not John domains. 
But the half-planes and the Koebe domain 
are John domains.
 
In this context we are interested to introduce the notion of {\em John functions}.
Motivation behind this comes from the definition of starlike and 
convex functions in $\D$. A starlike function is a conformal mapping of the unit disk
onto a domain starlike with respect to the origin and a convex function is 
one which maps the unit disk conformally onto a
convex domain.
For the theory of starlike and convex functions, we refer to the standard books \cite{Dur83,Goodman}. 
For analytic functions $f$ in $\D$, certain characterizations of John domains $f(\D)$
have been studied in \cite{COP96,Hag01}, where functions were not necessarily assumed to be normalized
and univalent (see for instance Lemma~\ref{cop-l2}).  
It is also interesting to see what changes would come in the situation when analytic functions are
normalized and univalent. This naturally leads to the concept of introducing John functions in $\D$.

\begin{definition}
A function $f\in \mathcal{S}$ is said to be a {\em John function}
\footnote{The authors wish to call these functions ``John functions''
in honor of Professor Fritz John.}
if $f(\D)$ is a John disk.
\end{definition}

Clearly, $f$ is bounded if and only if $f(\D)$ is a bounded John disk. We also call such functions 
the bounded John functions.
The functions $f_1(z)=(1/2){\rm Log}\,[(1+z)/(1-z)]$
and $f_2(z)=z/(1-z^2)$ respectively map the unit disk onto the parallel strip $D_1$
and the two-sided slit domain $D_2$. Since $D_1$ and $D_2$ are not John domains,
the functions $f_1$ and $f_2$ are not John functions. 
On the other hand, the functions $g_1(z)=z/(1-z)$ and $g_2(z)=z/(1-z)^2$ are John. 

We conclude this section with the following future directional work.

The famous analytical characterization of the starlike and convex functions are respectively
\begin{align}\label{StarCov}
{\rm Re}\,\left(\frac{zf'(z)}{f(z)}\right)>0 ~\mbox{ and }~
{\rm Re}\,\left(1+\frac{zf''(z)}{f'(z)}\right)>0, \quad z\in\D.
\end{align}
It is discussed above that neither convex nor starlike functions are 
necessarily John functions and also the other way around implication fails.
Therefore, although certain characterizations of John functions in different situations 
are studied 
in \cite{COP96,Hag01} (see also Lemma~\ref{cop-l2}), 
it would be interesting to find analytical characterizations of John functions 
similar to that of convex and starlike functions stated in (\ref{StarCov}). 

\vskip 1cm
\noindent
{\bf Acknowledgement.} The work of the first author was supported by University
Grants Commission, New Delhi (grant no. F.2-39/2011 (SA-I)). 
This research has been carried out from our earlier work when the 
first author was a PhD student at the Discipline of Mathematics, 
Indian Institute of Technology Indore. The authors would like to 
thank the referees for their careful reading of the previous versions of the paper and valuable remarks.

\end{document}